\documentclass{amsart}
\usepackage{amssymb}
\usepackage{mathrsfs}
\usepackage[francais,english]{babel}
\title[The Lind-Lehmer Constant for $\mathbb Z_2^r \times \mathbb Z_{4}^s$]{The Lind-Lehmer Constant for $\mathbb Z_2^r \times \mathbb Z_{4}^s$ }

\author{Michael J. Mossinghoff}\thanks{This work was supported in part by a grant from the Simons Foundation (\#426694 to M.~J. Mossinghoff).}
\address{Department of Mathematics \&  Computer Science \\Davidson College\\  Davidson, NC 28035-6996, USA}
\email{mimossinghoff@davidson.edu}

\author{Vincent Pigno}
\address{ Department of Mathematics \& Statistics\\
      California State University\\
          Sacramento, CA 95819, USA}
\email{vincent.pigno@csus.edu}

\author{Christopher Pinner}
\address{ Department of Mathematics\\
         Kansas State University\\ 
         Manhattan, KS 66506, USA}
\email{pinner@math.ksu.edu}
\keywords{Lind-Lehmer constant, Mahler measure}
\subjclass[2010]{Primary: 11R06; Secondary: 11B83, 11C08,  11G50, 11R09, 11T22, 43A40}
\date{\today}
\newcommand{\ve}{\varepsilon}

\newcommand{\Cal}{\mathcal}
\newcommand{\tsize}{\textstyle}
\newcommand{\be}{\begin{equation}}
\newcommand{\ee}{\end{equation}}
\newcommand{\ba}{\begin{align}}
\newcommand{\ea}{\end{align}}

\begin{document}



\begin{abstract}
We show that  the minimal positive logarithmic Lind-Mahler measure for a group of the form  $G=\mathbb Z_2^r\times\mathbb Z_4^s$ with $|G|\geq 4$ is $\frac{1}{|G|} \log (|G|-1).$
We also show that for $G=\mathbb Z_2 \times \mathbb Z_{2^n}$ with $n\geq 3$ this value is $\frac{1}{|G|} \log 9.$
Previously the minimal measure was only known for $2$-groups of the form $\mathbb Z_2^k$ or $\mathbb Z_{2^k}.$
\end{abstract}

\maketitle
\newtheorem{theorem}{Theorem}[section]
\newtheorem{corollary}{Corollary}[section]
\newtheorem{lemma}{Lemma}[section]
\newtheorem{conjecture}{Conjecture}[section]

\section{Introduction}\label{secIntro}

Recall that for a polynomial $F(x_1,\ldots ,x_k)$ in $\mathbb Z [x_1,\ldots ,x_k]$, one defines the traditional logarithmic 
Mahler measure by
$$ m(F) = \int_{0}^1 \!\cdots\! \int_0^1 \log |F(e^{2\pi i x_1},\ldots ,e^{2\pi i x_k})|\,dx_1\cdots dx_k. $$
In 2005, Lind \cite{Lind} viewed $[0,1]^k$ as the group $(\mathbb R/\mathbb Z)^k$ and  generalized the Mahler measure to arbitrary compact abelian groups.
In particular, for the finite abelian  group
$$ G=\mathbb Z_{n_1}\times \cdots \times \mathbb Z_{n_k} $$ 
and $F\in\mathbb Z [x_1,\ldots ,x_k]$, we define the \textit{logarithmic Lind-Mahler measure} by
$$  m_G(F) =\frac{1}{|G|} \sum_{x_1=1}^{n_1}\cdots \sum_{x_k=1}^{n_k} \log | F(e^{2\pi i x_1/n_1},\ldots ,e^{2\pi i x_k/n_k})|.
$$
The close connection to the group determinant is explored by Vipismakul \cite{Cid1}.
Writing
$$ w_n := e^{2\pi i/n},$$
we plainly have 
$$m_G(F) =\frac{1}{|G|} \log |M_G(F)|, $$
where
$$ M_G(F) := \prod_{j_1=1}^{n_1}\cdots \prod_{j_k=1}^{n_k} F\left(w_{n_1}^{j_1},\ldots ,w_{n_k}^{j_k}\right) $$
will be in $\mathbb Z$.
Analogous to the  classical  Lehmer problem \cite{Lehmer}, we can ask for the minimal $m_G(F)>0$, and to this end we define
the \textit{Lind-Lehmer constant} for $G$ by
$$ \lambda (G) := \min\{ |M_{G}(F)|>1 \; |\; F\in \mathbb Z [x_1,\ldots ,x_k]\}. $$
We use $|M_G(F)|$ rather than $m_G(F)$ or $|M_G(F)|^{1/|G|}$ so that we are dealing with integers; of course the minimal positive logarithmic measure will be $\frac{1}{|G|} \log \lambda (G)$.
As Lind observed, for $|G|\geq 3$ we always have the trivial bound
\be\label{triv} \lambda (G)\leq |G|-1, \ee
achieved, for example, by 
$$ F(x_1,\ldots ,x_k)=-1 +  \prod_{i=1}^k \left(\frac{{x_i}^{n_i}-1}{x_i-1}\right). $$  
Lind also showed that for prime powers $p^\alpha$ with $\alpha\geq1$ we have
\be \label{powerp}
\lambda (\mathbb Z_{p^\alpha})=\begin{cases} 3, & \hbox{ if $p=2$,} \\2, & \hbox{ if $p\geq 3$,} \end{cases} \ee
achieved with $x^2+x+1$ if $p=2$ and $x+1$ if $p\geq 3$.
Lind's results for cyclic groups were extended by Kaiblinger \cite{norbert} and Pigno \& Pinner \cite{pigno} so that $\lambda (\mathbb Z_m)$ is now known if $ 892\,371\,480\nmid m$. 
The value for the $p$-group $\mathbb Z_p^k$ was recently established by De Silva \& Pinner \cite{dilum}, but little is known for direct products involving at least one term $\mathbb Z_{p^{\alpha}}$ with $\alpha\geq2$.

Here we are principally interested  in the case of $2$-groups
\be \label{2group} G=\mathbb Z_{2^{\alpha_1}}\times \cdots \times \mathbb Z_{2^{\alpha_k}}. \ee
It was shown  in \cite{dilum} that for all $k\geq 2$
\be \label{all2s} \lambda (\mathbb Z_2^k)=2^k-1, \ee
a case of equality in \eqref{triv}.
We establish two main results regarding the Lind-Lehmer constant for groups of the form \eqref{2group}.
First, we prove that equality occurs in \eqref{triv} whenever $G$ is a $2$-group whose factors are all $\mathbb Z_2$ or $\mathbb Z_4.$

\begin{theorem}\label{thm1} If $G=\mathbb{Z}_2^r$ or  $\mathbb{Z}_4^s$ or $\mathbb{Z}_2^r\times\mathbb{Z}_4^s,$ then
$$ \lambda (G) =\max\{ 3, |G|-1\}. $$
\end{theorem}

Second, we show that this is not true for all $2$-groups: if we allow $\alpha_i\geq 3$ in \eqref{2group} then \eqref{triv} need not be sharp.

\begin{theorem}\label{thm2} For $n\geq 3$
$$ \lambda (\mathbb Z_2 \times \mathbb Z_{2^n}) = 9, $$
achieved with $F(x,y)=y^2+y+1$.
\end{theorem}

Crucial to our proofs of these statements will be a congruence satisfied by $M_G(F)$ when $G$ is a $p$-group.
This is a generalization of \cite[Lemma 2.1]{dilum} (see also \cite[Theorem 2.1.2]{Cid1}).

\begin{lemma}\label{cong}  If $p$ is a prime and  
\be \label{pgroup} G=\mathbb Z_{p^{\alpha_1}}\times \cdots \times \mathbb Z_{p^{\alpha_k}},\ee
 then
$$ M_G(F) \equiv F(1,\ldots ,1)^{|G|} \mod p^k. $$

\end{lemma}

Notice that for the $p$-group \eqref{pgroup} we have 
$$M_G(F)= \prod_{t_1=0}^{\alpha_1} \cdots \prod_{t_k=0}^{\alpha_k} N_{t_1,\ldots ,t_k}(F), $$
where 
$$N_{t_1,\ldots ,t_k}(F)= \hspace{-3ex} \prod_{\stackrel{j_1=1}{(j_1,p^{\alpha_1})=p^{t_1}}}^{p^{\alpha_1}} \hspace{-1.5ex} \cdots \hspace{-1.5ex} \prod_{\stackrel{j_k=1}{(j_k,p^{\alpha_k})=p^{t_k}}}^{p^{\alpha_k}} \hspace{-3ex} F(w_{p^{\alpha_1}}^{j_1},\ldots ,w_{p^{\alpha_k}}^{j_k})\in \mathbb Z.$$  
Since  $|1-w_{p^\alpha}^j|_p<1$ and the $N_{t_1,\ldots ,t_k}(F)$ are integers, we have 
$$N_{t_1,\ldots ,t_k}(F)\equiv F(1,\ldots ,1)^{\varphi(p^{\alpha_1-t_1})\cdots \varphi (p^{\alpha_k-t_k})} \text{ mod }p. $$
 In particular if $p\mid F(1,\ldots ,1)$ we have $p\mid N_{t_1,\ldots ,t_k}(F)$  for all $t_i$ and $|G|p^k\mid M_G(F)$. So, in view of \eqref{triv},  we can assume for the  $p$-group \eqref{pgroup}  that $p\nmid F(1,\ldots ,1)$ for any $F$ achieving $\lambda (G)$.

Thus, in the case of 2-groups we can assume an $F$ with minimal measure has  $F(1,\ldots ,1)$ odd, and by Lemma \ref{cong} we see that
\be \label{2cong}  M_G(F) \equiv 1 \text{ mod } 2^k. \ee
Note this immediately produces \eqref{all2s}.

Similarly for $3$-groups we can assume  that an $F$ with minimal measure has $3\nmid F(1,\ldots ,1)$ and $M_G(F) \equiv \pm 1 $ mod $3^k$.
This produces another case of equality in \eqref{triv}:
$$  \lambda (\mathbb Z_{3}^k)=3^k-1, $$
as observed in
\cite{dilum}.
For $G=\mathbb Z_3 \times \mathbb Z_{3^n}$, we have $M_G(F)\equiv \pm 1$ mod 9 and so we immediately obtain the minimal measure for an additional family of 3-groups.

\begin{theorem} \label{mod3}
For $n\geq 1$
$$ \lambda (\mathbb Z_3 \times \mathbb Z_{3^n}) = 8, $$
achieved with $F(x,y)=y+1$.
\end{theorem}

Section~\ref{secLemma} of this article is devoted to the proof of Lemma~\ref{cong}, Section~\ref{secThm1} establishes Theorem~\ref{thm1}, and Section~\ref{secThm2} proves Theorem~\ref{thm2}.

\section{Proof of Lemma \ref{cong}}\label{secLemma}

We proceed by induction on $\alpha_1 +\cdots +\alpha_k$. 
For $G=\mathbb Z_p$ we use that $|w_p-1|_p=p^{-1/(p-1)}<1$. Since $M_G(F)\in \mathbb Z$ and $M_G(F)\equiv F(1)^{p}$ mod $(1-w_p)$
we see that $M_G(F)\equiv F(1)^p$ mod $p$.

Set 
$$ g(x_1,\ldots ,x_k)= \prod_{l_1=1}^{p^{\alpha_1}} \cdots \prod_{l_k=1}^{p^{\alpha_k}} F(x_1^{l_1},\ldots ,x_k^{l_k}) $$
and let $I$ be the ideal in $\mathbb Z [x_1,\ldots ,x_n]$ generated by $x_1^{p^{\alpha_1}}-1,\ldots , x_k^{p^{\alpha_k}}-1.$ 
Expanding, we have 
$$ g(x_1,\ldots ,x_k)= \sum_{0\leq \ell_1 <p^{\alpha_1}} \cdots \sum_{0\leq \ell_k <p^{\alpha_k}}  a(\ell_1,\ldots ,\ell_k)x_1^{\ell_1}\cdots x_k^{\ell_k}\;\;  \text{ mod } I. $$
We set
\begin{align*}
S &:=  \sum_{j_1=1}^{p^{\alpha_1}} \cdots \sum_{j_k=1}^{p^{\alpha_k}}  g( w_{p^{\alpha_1}}^{j_1},\ldots ,w_{p^{\alpha_k}}^{j_k})  \\
&= \sum_{0\leq \ell_1 <p^{\alpha_1}} \cdots \sum_{0\leq \ell_k <p^{\alpha_k}}  a(\ell_1,\ldots ,\ell_k) \sum_{j_1=1}^{p^{\alpha_1}} \cdots \sum_{j_k=1}^{p^{\alpha_k}} w_{p^{\alpha_1}}^{j_1\ell_1}\cdots w_{p^{\alpha_k}}^{j_k\ell_k}\\
&= a(0,\ldots ,0)p^{\alpha_1 +\cdots +\alpha_k}. 
\end{align*}
If $(j_1,p^{\alpha_1})=\cdots =(j_k,p^{\alpha_k})=1$, then for these $\varphi (p^{\alpha_1})\cdots \varphi(p^{\alpha_k})$ values we have
$$   g( w_{p^{\alpha_1}}^{j_1},\ldots ,w_{p^{\alpha_k}}^{j_k}) = M_G(F). $$
Suppose that  $(j_1,p^{\alpha_1})=p^{t_1}, \ldots, (j_k,p^{\alpha_k})=p^{t_k}$  with at least one $t_j\neq 0$, and with $L\geq 0$ of the $t_i=\alpha_i$. Suppose without loss of generality that $t_i=\alpha_i$ for any $i=1,\ldots ,L$ and $t_i<\alpha_i$ for any $i=L+1,\ldots ,k$.
For these $\varphi (p^{\alpha_{L+1}-t_{L+1}})\cdots \varphi (p^{\alpha_{k}-t_{k}})$  values, applying the induction hypothesis to $G'=\mathbb Z_{p^{\alpha_{L+1}-t_{L+1}}}\times \cdots \times \mathbb Z_{p^{\alpha_k -t_k}}$,   we have
\begin{align*}    g( w_{p^{\alpha_1}}^{j_1},\ldots ,w_{p^{\alpha_k}}^{j_k}) &  =M_{G' }\left(F(1,\ldots ,1,x_{L+1},\ldots ,x_k)\right)^{p^{t_1+\cdots +t_k}} \\
 & = \left( F(1,\ldots ,1)^{p^{(\alpha_{L+1}-t_{L+1})+\cdots + (\alpha_k-t_k)}} +  hp^{k-L}\right)^{p^{t_1+\cdots +t_k}} \\
 & \equiv F(1,\ldots ,1)^{|G|} \text{ mod } p^{k-L+\alpha_1 +\cdots +\alpha_L+ t_{L+1}+\cdots +t_k}.
\end{align*}
Hence these $(p-1)^{k-L}\; p^{(\alpha_{L+1}-t_{L+1}-1)+\cdots +(\alpha_k-t_k-1)}$ terms contribute
$$ \varphi (p^{\alpha_{L+1}-t_{L+1}})\cdots \varphi (p^{\alpha_{k}-t_{k}})F(1,\ldots ,1)^{|G|}  \text{ mod } p^{\alpha_1+\cdots + \alpha_k} $$
to $S$. 
Thus
\begin{align*}  0  & \equiv \varphi(p^{\alpha_1})\cdots \varphi(p^{\alpha_k})M_G(F) +  \left(p^{\alpha_1+\cdots +\alpha_k}- \varphi(p^{\alpha_1})\cdots \varphi(p^{\alpha_k})\right) F(1,\ldots ,1)^{|G|}  \\
 & \equiv (p-1)^{k}p^{\alpha_1+\cdots +\alpha_k -k}\left(M_G(F)-F(1,\ldots ,1)^{|G|}\right)  \text{ mod } p^{\alpha_1+\cdots +\alpha_k} \end{align*}
and the statement follows.\qed

\section{Proof of Theorem \ref{thm1}}\label{secThm1}

To prove Theorem~\ref{thm1}, we require the following lemma.

\begin{lemma} \label{vanishing} Suppose that $F\in \mathbb Z[x_1,\ldots ,x_n]$, and let $I$ denote the ideal of $\mathbb Z[x_1,\ldots ,x_n]$ 
generated by $x_1^{n_1}-1,\ldots ,x_k^{n_k}-1$.  Then  $F(w_{n_1}^{j_1},\ldots ,w_{n_k}^{j_k})=0$ for 
all $1\leq j_i\leq n_i$ if and only if $F\in I$.
\end{lemma}

\begin{proof}
Plainly any $F$ in $I$ will have $F\left(w_{n_1}^{j_1},\ldots,w_{n_k}^{j_k}\right)=0$ for all $0\leq j_i < n_i$.
Conversely, suppose that $F\left(w_{n_1}^{j_1},\ldots ,w_{n_k}^{j_k}\right)=0$
for all $0\leq j_i < n_i$. Clearly any $F$ can be reduced mod $I$ to
a polynomial of degree less than $n_i$ in each $x_i$:
$$
F(x_1,\ldots ,x_k) = \sum_{t_1=0}^{n_1-1}\cdots \sum_{t_k=0}^{n_k-1} a(t_1,\ldots , t_k)x_1^{t_1}\cdots x_k^{t_k} \text{ mod } I.
$$
Since $\sum_{j_i=0}^{n_i-1} w_{n_i}^{(t_i-T_i)j_i}=0$ if
$t_i\not\equiv T_i$ mod $n_i$ (and $n_i$ otherwise) we have
$$  a(T_1,\ldots, T_k)=\frac{1}{|G|}\sum_{j_1=0}^{n_1-1}\cdots
\sum_{j_k=0}^{n_k-1} F\left(w_{n_1}^{j_1},\ldots
,w_{n_k}^{j_{k}}\right) w_{n_1}^{-T_1j_1}\cdots w_{n_k}^{-T_kj_k}.
$$
So $a(T_1,\ldots, T_k)=0$ for all $0\leq T_i<n_i$ and  $F=0$ mod
$I$.
\end{proof}

We now proceed to the proof of our first principal result.

\noindent
\begin{proof}[Proof of Theorem~\ref{thm1}]
Suppose that $G=\mathbb Z_{2^{\alpha_1}}\times \cdots \times \mathbb Z_{2^{\alpha_k}}$ with $2^{\alpha_i}=4$ for $1\leq i\leq s$ and $2^{\alpha_i}=2$ for $s+1\leq i\leq k$.
We write $r=k-s$.
In view of \eqref{powerp} and \eqref{all2s} we may assume that $k\geq 2$ and $s\geq 1$.
Suppose that $F(x_1,\ldots ,x_k)$ has 
$$1< |M_G(F)|< |G|-1=2^{k+s}-1. $$
Suppose that $F(x_1,\ldots,x_k)$ is a non-unit with at least one of the $x_j$ complex, say  $x_1=\pm i$, and set $G'=\mathbb Z_{2^{\alpha_2}}\times \cdots \times \mathbb Z_{2^{\alpha_k}}$. 
Plainly we may write 
$$ M_G(F) = AB, $$
with
$$ A:=M_{\mathbb Z_{2}\times G'}(F),\;\;\; B:= M_{G'} (F(i,x_2,\ldots,x_k)F(-i,x_2,\ldots ,x_k)). $$
From \eqref{2cong} we know  that $M_G(F)$ and  $A$, and hence $B$,  are all congruent to 1 mod $2^{k}$. Also $B$ will be of the form $|a+ib|^2$ and hence cannot be negative. Since it contains a non-unit we have $B>1$, hence $B\geq 2^k+1$. 
If $A\neq 1$ then $|A|\geq 2^k-1$ and $|M_G(F)|\geq (2^k-1)(2^k+1) = 4^k-1\geq |G|-1$, so we must have 
$A=1$.
Thus if $F(x_1,x_2,\ldots ,x_k)$ is a non-unit with $x_j=\pm i$, then we may assume $F(y_1,\ldots ,y_k)$ is a unit if $y_j=\pm 1$.
We have two possibilities:

\begin{itemize}
\item[] Case (a). There is at least one non-unit $F(x_1,\ldots ,x_k)$  with some $x_j=\pm i$.
\item[] Case (b). $F(x_1,\ldots ,x_k)$ is a unit whenever any of the $x_j=\pm i$.
\end{itemize}

With $I$ denoting the ideal generated by the $x_j^{2^{\alpha_j}}-1$, and splitting the $x_1$ dependence into even and odd exponents $p(x_1)=\alpha(x_1^2)+x_1\beta (x_1^2)$, we can write
$$  F(x_1,\ldots ,x_k) = \hspace{-4ex} \sum_{\substack{0\leq \ve_2,\ldots ,\ve_s\leq 3,\\  0\leq\ve_1, \ve_{s+1},\ldots ,\ve_k\leq 1} } \hspace{-4ex}a(\ve_1,\ve_2,\ldots ,\ve_k)(x_1^2) \; \;x_1^{\ve_1}x_2^{\ve_2}\cdots x_k^{\ve_k} \;\; \text{ mod } I.$$
Since $F(1,\ldots ,1)=\sum a(\ve_1,\ve_2,\ldots ,\ve_k)(1)$ is odd, we know that at least one of the $a(\ve_1,\ve_2,\ldots ,\ve_k)(1)$ is odd.
Replacing $F$ by $x_1^{\delta_1}\cdots x_n^{\delta_n}F$ with 
 $0\leq \delta_1,\delta_{s+1},\ldots ,\delta_k \leq 1$ and $0\leq \delta_2,\ldots ,\delta_s\leq 3$, and reducing mod $I$,
we can reshuffle the $a(\ve_1,\ldots ,\ve_k)(x_1^2)$ and assume that $a(0,\ldots ,0)(1)$ is odd. Replacing $F$ by $-F$ we can assume that $F(1,\ldots ,1)\equiv 1 $ mod 4.

\vspace{1ex}
{\bf Case (a)}. Suppose we have non-units with complex $x_j$. Reordering and taking $x_j\mapsto \pm x_1x_j$ for $2\leq j\leq s$ and $x_j\mapsto \pm x_j$ for $s<j\leq k$ as necessary, we assume that the first of these is $\gamma_1=F(i,1,\ldots ,1)$.
If (after the transformations) there are other non-units with complex entries in positions other than the first, by reordering 
and substituting  $x_j\mapsto x_jx_2$ as necessary for $j\geq 3$, we may assume that $\gamma_2=F(\pm i,i,\pm 1,\ldots ,\pm 1)$.
We repeat this $1\leq h\leq s$ times until we have $h$ non-units $\gamma_j=F(a_{j1},\ldots ,a_{jk})$ 
with $a_{jj}=i$, $a_{j\ell}=\pm i$ for $1\leq \ell <j$ and $a_{j\ell}=\pm 1$ for $h<\ell \leq k$, and $F(x_1,\ldots ,x_k)$ is a unit 
whenever  $x_{\ell}=\pm i$ with $h<\ell \leq s$ if $h<s$. 

Since the  $F(\pm 1,x_2,\ldots ,x_k)$ are all units,  with $F(1,\ldots ,1)=1$, and 
$$ a(0,\ldots ,0)(1) = \frac{2}{|G|} \hspace{-5ex} \sum_{\substack{\vspace{2pt}\\x_2,\ldots ,x_s=\pm i,\pm 1 \\ x_1,x_{s+1},\ldots ,x_k=\pm 1}} \hspace{-5ex} F(x_1,\ldots ,x_k)$$
 is odd,  plainly the $F(\pm 1,x_2,\ldots ,x_k)$  must all be 1. 
 Applying Lemma \ref{vanishing}, we may therefore assume that 
$$  F(x_1,\ldots ,x_k) = 1+(x_1^2-1) \hspace{-5ex} \sum_{\substack{0\leq \ve_2,\ldots ,\ve_s\leq 3,\\  0\leq \ve_1,\ve_{s+1},\ldots ,\ve_k\leq 1} } \hspace{-6ex}a(\ve_1,\ve_2,\ldots ,\ve_k)x_1^{\ve_1}x_2^{\ve_2}\cdots x_k^{\ve_k}. $$
Notice that the $F(\pm i, x_2,\ldots ,x_k)\in \mathbb Z[i]$ will all have odd real part and even imaginary  part. Moreover, 
 writing $u=(1-i)$  where  $u^2\mid 2$ and $x_j\equiv 1$ mod $u$ for  any $x_j=\pm 1$ or $\pm i$, the $F(\pm i, x_2,\ldots ,x_k)$  must all  be congruent  mod $u^3$ in $\mathbb Z[i]$. Since  $|u|_2=2^{-1/2}$  plainly two units $\pm 1,\pm i,$  in $\mathbb Z [i]$ cannot be congruent mod $u^3$ unless they are equal.
If $h\geq 2$ then we know that the $F(\pm i,\pm 1,x_3,\ldots ,x_k)$ will all be units and so must be all 1 or all $-1$. Replacing $F$ by $x_1^2F$ we can assume that they are all 1. Applying Lemma \ref{vanishing} we get 
$$ F(x_1,\ldots ,x_k)=1+(x_1^2-1)(x_2^2-1) \hspace{-7ex} \sum_{\substack{0\leq \ve_3,\ldots ,\ve_s\leq 3,\\  0\leq \ve_1,\ve_2,\ve_{s+1},\ldots ,\ve_k\leq 1} } \hspace{-7ex} a(\ve_1,\ve_2,\ldots ,\ve_k)x_1^{\ve_1}x_2^{\ve_2}\cdots x_k^{\ve_k}. $$
Likewise, if $h\geq 3$ we have  that $F(\pm i,\pm i,\pm 1,x_4,\ldots ,x_k)$ are all units and 1 mod 4, so these must all equal 1. Applying the lemma and repeating up to $F(\pm i,\ldots ,\pm i,\pm 1,x_{h+1},\ldots ,x_k)$, we deduce that
$$  F(x_1,\ldots ,x_k)=1+\prod_{j=1}^h (x_j^2-1) \hspace{-8ex} \sum_{\substack{0\leq \ve_{h+1},\ldots ,\ve_s\leq 3,\\  0\leq \ve_1,\ldots ,\ve_h,\ve_{s+1},\ldots ,\ve_k\leq 1} } \hspace{-8ex} a(\ve_1,\ve_2,\ldots ,\ve_k)x_1^{\ve_1}x_2^{\ve_2}\cdots x_k^{\ve_k}. $$
If $s>h$, we further have that the  $F(\pm i,\ldots ,\pm i, x_{h+2},\ldots ,x_k)$ are all units. If $h\geq 2$ they will all be 1 mod 4 
and so must all equal 1.
If $h=1$ then they are all 1 or all $-1$ and, by replacing $F$ by $x_1^2F$ if necessary, we may assume they are all $1$. 
Separating into real and imaginary parts, applying Lemma~\ref{vanishing}, then repeating for each variable, we find
$$  F(x_1,\ldots ,x_k)=1+\prod_{j=1}^h(x_j^2-1)\prod_{j=h+1}^s(x_j^2+1) \hspace{-3ex} \sum_{ 0\leq \ve_1,\ldots ,\ve_k\leq 1} \hspace{-3ex} a(\ve_1,\ve_2,\ldots ,\ve_k)x_1^{\ve_1}x_2^{\ve_2}\cdots x_k^{\ve_k}. $$

Suppose that there are $t\geq 1$  conjugate pairs of non-units $F(a_{j1},\ldots ,a_{jk})=\gamma_j$.
Then plainly
\be
\label{congL} \gamma_j=a_j+ib_j,\;\;\; a_j\equiv 1 \text{ mod } 2^s,\;\; b_j\equiv 0 \text{ mod } 2^s. \ee
Trivially we have $|\gamma_j|^2\geq 5$, and if $t\geq r+s$ then
$$  |M_G(F)|  \geq 5^t \geq 5^r\cdot5^s > 2^r\cdot4^s - 1,$$
so we can assume that 
\be \label{ineq1} t\leq r+s-1. \ee
If $t\leq r$ then, by using the transformation $x_{\ell}\mapsto x_{\ell}x_j$  if $x_j=-1$ to remove $x_{\ell}=-1$ with $\ell >j,$ we can 
assume that the $r$-tuples $(x_{s+1},\ldots ,x_k )$  achieving the $\gamma_j$ take the form 
$$(1,\ldots ,1), (\pm 1,1,\ldots ,1), (\pm 1,\pm 1,1,\ldots ,1), \ldots , (\underbrace{\pm 1,\ldots ,\pm 1}_{t -1},1,\ldots ,1). $$
In particular, $F(x_1,\ldots ,x_k)$ will be a unit if $x_j=-1$ for any $s+t\leq j\leq k$.
(If $s\geq 2$ the units  will all be 1; if $s=1$ we may need to take $x_1^2F$ to make the value when  $x_{s+t}=-1$ and hence the rest equal 1.)
Successively applying the lemma again, we find 
$$ F(x_1,\ldots ,x_k)=1+\prod_{j=1}^h(x_j^2-1)\prod_{j=h+1}^s(x_j^2+1) \prod_{j=s+t}^{k} (x_j+1) R $$
with
$$ R= \hspace{-3ex} \sum_{ 0\leq \ve_1,\ldots ,\ve_{s+t-1}\leq 1} \hspace{-4ex} a(\ve_1,\ve_2,\ldots ,\ve_{s+t-1})x_1^{\ve_1}x_2^{\ve_2}\cdots x_{s+t-1}^{\ve_{s+t-1}}. $$
Hence we obtain that
$$ \gamma_j=a_j+ib_j,\;\;\; a_j\equiv 1 \text{ mod  } 2^{s+r+1-t},\;\; b_j\equiv 0 \text{ mod } 2^{s+r+1-t}. $$
 From \eqref{ineq1} and \eqref{congL} this is plainly also valid if $t>r$. 
Thus, we have 
$$ |  M_G(F)| =|\gamma_1|\cdots |\gamma_t|\geq (2^{r+s+1-t}-1)^{2t} > 2^{2t(r+s+.5 -t)} \geq 2^{2(r+s-.5)} \geq 2^{r+2s} $$
for $r\geq 1$.
If $r=0$ and $t\geq 2$ then we have $s\geq 2$, and from \eqref{congL} we obtain
$$ |M_G(F)|\geq (2^s-1)^{2t}> 2^{2t(s-0.5)} \geq 2^{4s-2}\geq 4^s. $$
Finally if $t=1$ and $r=0$ then, since $F(i,1,\ldots 1)$ and its conjugate are the only non-units, we know that 
$F(\pm i,-1,x_3,\ldots ,x_k)$ are all units and so equal 1.
Hence we can add an extra factor $(x_2+1)$ to get
$$ |M_G(F)|\geq (2^{s+1}-1)^2> 2^{2s}. $$

\vspace{1ex}
{\bf Case (b)}.  Since $a(0,\ldots ,0)(1)$ is odd, we know that $a(0,\ldots ,0)(-1)$ is odd.
Since the  $F(\pm i,x_2,\ldots ,x_k)$ are all units  and 
$$ a(0,\ldots ,0)(-1) = \frac{1}{|G|/2} \hspace{-4ex} \sum_{\substack{x_1=\pm i \\ x_2,\ldots ,x_s=\pm i,\pm 1 \\ x_{s+1},\ldots ,x_k=\pm 1}} \hspace{-4ex} F(x_1,\ldots ,x_k)$$
 is odd,  plainly the $F(\pm i,x_2,\ldots ,x_k)$  must all be 1 or all be $-1$. Replacing $F$ by $x_1^2F$ we assume $F(\pm i,x_2, \ldots,x_k)=1.$
Applying Lemma \ref{vanishing} to the real and imaginary parts we can assume that
$$ F(x_1,\ldots ,x_k)=1+(x_1^2+1) \hspace{-5ex} \sum_{\substack{0\leq \ve_2,\ldots ,\ve_s\leq 3,\\  0\leq \ve_1,\ve_{s+1},\ldots ,\ve_k\leq 1} } \hspace{-5ex} a(\ve_1,\ve_2,\ldots ,\ve_k)x_1^{\ve_1}x_2^{\ve_2}\cdots x_k^{\ve_k}. $$
Notice that all the $F(\pm 1,x_2,\ldots ,x_k)\equiv F(1,\ldots ,1)\equiv 1$ mod $u^3$. Hence if $s>1$ the 
units $F(\pm 1,\pm i,x_3,\ldots ,x_k)$ are all 1.
Applying the Lemma and repeating we obtain
$$ F(x_1,\ldots ,x_k)=1+\prod_{j=1}^s (x_j^2+1) \hspace{-3ex} \sum_{0\leq \ve_1,\ldots ,\ve_k\leq 1} \hspace{-3ex} a(\ve_1,\ve_2,\ldots ,\ve_k)x_1^{\ve_1} x_2^{\ve_2}\cdots x_k^{\ve_k}. $$
Hence we have 
$$ M_G(F) = M_{\mathbb Z_2^k} (f) $$
where 
$$ f(x_1,\ldots ,x_k)=1+ 2^s \hspace{-3ex} \sum_{0\leq \ve_1,\ldots ,\ve_k\leq 1} \hspace{-3ex} A(\ve_1,\ldots ,\ve_k) x_1^{\ve_1}\cdots x_k^{\ve_k}.  $$ 
Suppose that there are $t$ elements  $f(\pm 1,\ldots ,\pm 1)$ that are not $\pm 1$. 
If $t\geq k+s-1$ then plainly $|M_G(F)|\geq 3^t \geq 3^{k+s-1}> 2^{k+s}$ since $k+s\geq 3$, so we assume that $t\leq k+s-2$.
Sending $x_j\mapsto -x_j$ we
assume that one of them is $f(1,\ldots,1)=\gamma_1$. If $t>1$ then, reordering and mapping $x_{\ell}\mapsto x_{\ell}x_{j}$ if we have $\ell > j$ with $x_{\ell}=x_j=-1$, we can assume that the remaining values are
$\gamma_2=f(-1,1,\ldots ,1), \gamma_3=f(a_{31},a_{32},1,\ldots ,1), \ldots ,\gamma_t=f(a_{t1},\ldots a_{t(t-1)},1,\ldots ,1)$. If 
$t\leq k$ then we will have $f(x_1,\ldots, x_k)=1$ whenever $x_j=-1$ for some $t\leq j \leq k$, and applying the lemma we find
$$ f(x_1,\ldots ,x_k)=1+ 2^s\prod_{j=t}^{k} (x_j+1) \hspace{-4ex} \sum_{0\leq \ve_1,\ldots ,\ve_{t-1}\leq 1} \hspace{-4ex} A(\ve_1,\ldots ,\ve_{t-1}) x_1^{\ve_1}\cdots x_{t-1}^{\ve_{t-1}}.  $$ 
Thus the
$$ \gamma_j \equiv 1 \text{ mod } 2^{s+k-t+1} $$
(with this trivially holding if $k\leq t-1$), and
$$ |M_G(F)|\geq (2^{s+k+1-t}-1)^t.  $$
For $t=1$ this gives
$$  |M_G(F)|\geq 2^{s+k}-1 =|G|-1, $$ 
and for $t\geq 2$ 
\[
|M_G(F)| \geq 2^{t(s+k+0.5-t)} \geq  2^{2s+2k-3} \geq 2^{s+k}.\qedhere
\]

\end{proof}

\section{Proof of Theorem \ref{thm2}}\label{secThm2}

Using $\Phi_j(x)$ to denote the $j$th cyclotomic polynomial and recalling (see \cite{Apostol} or \cite{ELehmer}) that for $j>k$ the resultant
$| \text{Res}(\Phi_j,\Phi_k)|=q^{\varphi(k)}$ if $j=kq^{\alpha}$ for some prime $q$ and 1 otherwise, we see that
$$ M_{\mathbb Z_2\times \mathbb Z_{2^n}}(1+y+y^2) =M_{\mathbb Z_{2^n}}(\Phi_3(y))^2= \left(\prod_{j=0}^n \left|\text{Res}(\Phi_3,\Phi_{2^j})\right|\right)^2=9. $$ 
Let $G=\mathbb Z_2 \times \mathbb Z_{2^n}$.
Reducing mod $x^2-1$, we can write our $F(x,y)$ in $\mathbb Z[x,y]$ in the form
$$ F(x,y)=A_0(y^2) + xA_1(y^2)+yA_2(y^2)+xyA_3(y^2). $$ 
Plainly,
$$ M_G(F(x,y))=M_{\mathbb Z_{2^n}} (F(1,y)) M_{\mathbb Z_{2^n}} (F(-1,y)), $$
where each of these measures is a product of $n+1$ integers,
$$M_{\mathbb Z_{2^n}} (f(y))=\prod_{j=0}^n N_j(f), \;\;\; N_j(f) :=\text{Res}(f,\Phi_{2^j}),$$ 
that is,
$$  N_0(f)=f(1),\;\;\; N_1(f)=f(-1),\;\;N_2(f)=f(i)f(-i)=|f(i)|^2, $$
and, writing $w_j:=e^{2\pi i/2^j},$  for any $j=3,\ldots,n$, we have
$$ N_j(f)= \prod_{\stackrel{k=1}{k \text{ odd}}}^{2^j}  f(w_{j}^k)= \prod_{\stackrel{k=1}{k \text{ odd}}}^{2^{j-1}}  f(w_{j}^k)f(-w_j^k)=  \left|R_j(f)\right|^2,$$ 
where
$$  R_j(f):= \hspace{-2ex} \prod_{\stackrel{k=1}{k \equiv 1 \text{ mod }4}}^{2^{j-1}} \hspace{-2ex} f(w_{j}^k)f(-w_j^k) \in \mathbb Z [i],\;\; \;\;3\leq j\leq n. $$
Note $N_j(f)$ and  $R_j(f)$ represent the norms of $f(w_j^k)$ from $\mathbb Q(w_j)$ to $\mathbb Q$ and $\mathbb Q(i)$ respectively, and since they are algebraic integers they will be in $\mathbb Z$ and  $\mathbb Z [i]$, respectively.

Since $|1-w_j|_2=2^{-1/\varphi(2^j)}$, each $N_j(F(\pm 1,y))\equiv F(1,1)^{2^{j-1}}$ mod 2, and if $M_G(F)< 2^{2n+2}$ we can assume $F(1,1)$ and all the $N_j(F(\pm 1,y))$ are odd. Note that for all the $j\geq 2$ we have $N_j(F(\pm 1,y))=|a+ib|^2=a^2+b^2\equiv 1$ mod 4.

 If $|M_G(F)|<9$ then 
$|M_{\mathbb Z_{2^n}}(F(1,y))|$ or $|M_{\mathbb Z_{2^n}}(F(-1,y))|$ must be 1. Replacing $x\mapsto -x$   as necessary we assume that 
$$ 1<| M_{\mathbb Z_{2^n}}(F(1,y))|<9,\;\;\; |M_{\mathbb Z_{2^n}}(F(-1,y))|=1. $$

Since 
$$F(1,1)=A_0(1)+A_1(1)+A_2(1)+A_3(1)$$ 
is odd, we can assume that at least one of the $A_i(1)$ is odd. Replacing $F$ by $xF$ or $yF$ or $xyF$ and reducing by $x^2-1$ as necessary, we may assume that $A_0(1)$ is odd. Replacing $y$ by $-y$ and $F$ by $-F$ as necessary, we may further assume that $|F(1,1)|\geq| F(1,-1)|$ and $F(1,1)>0$. 

Since 
\begin{align*}
 F(1,-1) & =A_0(1)+A_1(1)-A_2(1)-A_3(1), \\
F(-1,1)& =A_0(1)-A_1(1)+A_2(1)-A_3(1), \\
F(-1,-1)& =A_0(1)-A_1(1)-A_2(1)+A_3(1), \end{align*}
we have
\begin{align*}A_0(1) & =\frac{1}{4} (F(1,1)+F(1,-1)+F(-1,1)+F(-1,-1)),\\
 A_1(1) & =\frac{1}{4}(F(1,1)+F(1,-1)-F(-1,1)-F(-1,-1)), \\
A_2(1) & =\frac{1}{4}(F(1,1)-F(1,-1)+F(-1,1)-F(-1,-1)), \\
A_3(1) & =\frac{1}{4}(F(1,1)-F(1,-1)-F(-1,1)+F(-1,-1)).
\end{align*}
Observe that
\[
F(1,w_j^k)F(1,-w_j^k) = \left(A_0(w_j^{2k})+A_1(w_j^{2k})\right)^2-w_j^{2k}\left(A_2(w_j^{2k})+A_3(w_j^{2k})\right)^2
\]
and
\[
F(-1,w_j^k)F(-1,-w_j^k) = \left(A_0(w_j^{2k})-A_1(w_j^{2k})\right)^2-w_j^{2k}\left(A_2(w_j^{2k})-A_3(w_j^{2k})\right)^2
\]
differ by
$$ 4\left(A_0(w_j^{2k})A_1(w_j^{2k})-w_j^{2k}A_2(w_j^{2k})A_3(w_j^{2k}) \right)\in 4\mathbb Z [w_{j-1}].$$
Hence $ R_j(F(1,y))$ and $R_j(F(-1,y))$ differ by an element of $ 4\mathbb Z [w_{j-1}]$ and, since both are in $\mathbb Z [i]$, we conclude that
$$ R_j(F(1,y))-R_j(F(-1,y)) \in 4\mathbb Z[i]. $$
Since $N_j(F(-1,y))=1$, we have $R_j(F(-1,y))=\pm 1$ or $\pm i$, and either $R_j(F(1,y))=R_j(F(-1,y))$ and $N_j(F(1,y))=1$, or $N_j(F(1,y))\geq (4-1)^2=9$.

Thus if $|M_G(F)|<9$ then we must have $N_j(F(1,y))=N_j(F(-1,y))=1$ for $j=3,\ldots ,n$ and
$ M_G(F)=M_{\mathbb Z_2 \times \mathbb Z_4}(F). $
By Theorem \ref{thm1} and Lemma \ref{cong}, we have $|M_{\mathbb Z_2 \times \mathbb Z_4}(F)|\geq 7$ and  $M_{\mathbb Z_2 \times \mathbb Z_4}(F)\equiv 1$ mod 4,  and so 
$$ M_G(F)=M_{\mathbb Z_2 \times \mathbb Z_4}(F)=-7. $$
Since $N_j(f)\equiv 1$ mod 4 for $j\geq 2$ we must have $|F(1,1)F(1,-1)|=7$ and $N_2(F(1,y))=1$  and
$$ F(1,1)=7,\;\;  F(1,-1),F(-1,\pm 1)=\pm 1,\;\; F(\pm 1,\pm i)=\pm 1 \hbox{ or } \pm i, $$ 
with $ R_{j}(F(1,y))=R_j(F(-1,y))=\pm 1 \hbox{ or } \pm i$  for $j=3,\ldots ,n$.

We have
$$ A_0(1) =\frac{1}{4} (F(1,1)+F(1,-1)+F(-1,1)+F(-1,-1))=\frac{1}{4} (7\pm 1\pm 1\pm 1) $$
and, since $A_0(1)$ is odd,  we   must have $F(1,-1)=F(-1,\pm 1)=-1$ and $A_0(1)=1$ and $A_1(1)=A_2(1)=A_3(1)=2$.
Hence
$$ F(x,y) = 1+2x+2y+2xy +(y^2-1)(B_0(y^2) + xB_1(y^2)+yB_2(y^2)+xyB_3(y^2)). $$
Thus
\begin{align*}
F(1,i)  & =  3+4i -2(B_0(-1)+B_1(-1)+ iB_2(-1)+iB_3(-1)),     \\
F(-1,i)  & = -1 -2(B_0(-1)-B_1(-1)+ iB_2(-1)-iB_3(-1)),
\end{align*}
and since $F(\pm 1,i)$ are units with odd real part and difference  in $4\mathbb Z [i]$ they must be both be 1 or $-1$.
By replacing $F$ by $y^2F$ as necessary, we may assume $F(\pm 1,i)=-1$.
Solving, we obtain $B_0(-1)=B_{1}(-1)=B_2(-1)=B_3(-1)=1$ and
$$ F(x,y)= -1 +  (1+x)(1+y)(1+y^2) + (y^4-1)\left(   C_0(y^2) + xC_1(y^2)+yC_2(y^2)+xyC_3(y^2)\right).  $$
Therefore
\[
F(1,w_3)F(1,-w_3) = (1+2i-2C_0(i)-2C_1(i))^2 -4i(1+i-C_2(i)-C_3(i))^2
\]
and
\[
F(-1,w_3)F(-1,-w_3) = (-1-2C_0(i)+2C_1(i))^2 -4i(C_2(i)-C_3(i))^2.
\]
Since both are units and are members of $1+4\mathbb Z[i]$, these must both equal 1.
However, their difference
$$ 4\Big( (i-2C_0(i))(1+i-2C_1(i))-i(1+i-2C_3(i))(1+i-2C_2(i)) \Big) \in 4(1+i +2\mathbb Z [i]) $$
is not zero.\qed

\end{document}